\documentclass{amsart}

\usepackage{latexsym}
\usepackage{verbatim}

\usepackage{amssymb,amsmath,amscd,graphicx,color,enumerate}

\usepackage{hyperref}

 \usepackage[normalem]{ulem}
\renewcommand\sout{\bgroup\markoverwith
 {\textcolor{red}{\rule[0.7ex]{3pt}{1.4pt}}}\ULon}

\usepackage[font=bf]{caption}

\usepackage{microtype}

\usepackage[all]{xy}

 \usepackage{color}                

\definecolor{darkblue}{rgb}{0.2, 0., .7}
\newcommand\op[1]{\psi^{#1}(M; E)}
\newcommand{\clop}{\overline{\psi^{0}}(M; E)}
\newcommand{\clopn}{\overline{\psi^{-1}}(M; E)}

\newcommand{\Hom}{\operatorname{Hom}}
\newcommand{\End}{\operatorname{End}}

\newcommand{\CC}{\mathbb C}

\newcommand{\NN}{\mathbb N}

\newcommand{\RR}{\mathbb R}

\newcommand{\maB}{\mathcal B}
\newcommand{\maC}{\mathcal C}

\newcommand{\maG}{\mathcal G}
\newcommand{\maH}{\mathcal H}
\newcommand{\maK}{\mathcal K}
\newcommand{\maL}{\mathcal L}

\newcommand{\maN}{\mathcal N}
\newcommand{\maP}{\mathcal P}
\newcommand{\maQ}{\mathcal Q}

\newcommand\ede{\, := \,}
\newcommand\seq{\, = \,}

\newcommand\one{\mathbf{1}}

\newcommand{\oid}{\operatorname{Id}}

\newtheorem{theorem}{Theorem}[section]
\newtheorem{lemma}[theorem]{Lemma}
\newtheorem{proposition}[theorem]{Proposition}
\newtheorem{corollary}[theorem]{Corollary}

\theoremstyle{definition}
\newtheorem{definition}[theorem]{Definition}
\newtheorem{remark}[theorem]{Remark}
\newtheorem{example}[theorem]{Example}

\newtheorem{notation}[theorem]{Notation}

\newtheorem{NotHyp}[theorem]{Notation and hypothesis}

\newtheorem*{proposition*}{Proposition}
\newtheorem*{definition*}{Definition}
\newtheorem*{theorem*}{Theorem}

\newcommand{\Sp}{\mathrm{Sp}}

\begin{document}

\title[Simonenko's local principle and Fredholm conditions]{
A GENERAL SIMONENKO LOCAL
PRINCIPLE and FREDHOLM CONDITION FOR ISOTYPICAL
COMPONENTS
}
  
\author[A. Baldare]{Alexandre Baldare}
\email{alexandre.baldare@math.uni-hannover.de}
 \address{Institut für Analysis, Welfengarten 1, 30167 Hannover, Germany}
\urladdr{https://baldare.github.io/Baldare.Alexandre/}



%
%

\begin{abstract}

In this paper, we derive, from a general Simonenko's local principle, 
Fredholm criteria for restriction to isotypical components. 
More precisely, we give a full proof, of the equivariant local principle 
for restriction to isotypical components
of invariant pseudodifferential operators announced in \cite{BCLN2}.  
Furthermore, we extend this result 
by relaxing the hypothesis made in the preceding quoted paper. 
\end{abstract}

\maketitle \tableofcontents

\section*{Introduction}

This paper is devoted to the proof and an application of a general Simonenko's local principle 
to $G$-invariant operators on closed manifolds. 
Local principles first appeared in Simonenko's work \cite{simonenko1965new} 
and more general forms appeared in \cite{allan,douglas,Roch.book2}.
Since then local principles were intensively  used to obtain Fredholm condition for singular operators, see for examples 
\cite{bottcher,Kisil_2012,Schulze.localPrinciple,semenyuta,vasilyev2018pseudo} and the references therein.
As a consequence of the general Simonenko's local principle, we derive Fredholm conditions 
for restriction of $G$-invariant pseudodifferential operators to isotypical components.

Let $G$ be a compact Lie group and denote by $\widehat{G}$ the set of isomorphism 
classes of irreducible unitary representations of $G$. If $P : \maH \rightarrow \maH'$ is a $G$-invariant 
continuous linear map between Hilbert spaces and $\alpha \in \widehat{G}$, 
then the operator $P$ induces a well defined continuous linear map between the $\alpha$-isotypical components
\begin{equation*}
\pi_\alpha(P) : \maH_\alpha \rightarrow \maH'_\alpha.
\end{equation*}

In this paper, we are interested in the case where $P$ is a pseudodifferential operator
acting between sections of two vector bundles.

Assume that our compact Lie group $G$ acts smoothly and isometrically 
on a compact Riemannian manifold $M$ and on two hermitian vector bundles $E_0$ and $E_1$. 
Furthermore, let $P : \maC^\infty(M,E_0) \rightarrow \maC^\infty(M,E_1)$
be a $G$-invariant, classical, order $m$, pseudodifferential operator on $M$.
Since $P$ is $G$-invariant, its principal symbol $\sigma_m(P)$  belongs
 to $ \maC^{\infty}(T^*M \smallsetminus \{0\}; \Hom(E_0, E_1))^G$. 
Let $G_\xi$  and $G_x$ denote the isotropy subgroups of $\xi \in
T_x^*M$ and $x \in M$, as usual. Then $G_\xi \subset G_x$ acts linearly
on the fibers $E_{0x}$ and $E_{1x}$. 
Following 
\cite{atiyahGelliptic}, denote by 
$T^*_GM$ the {\em $G$-transverse cotangent space}, see Equation \eqref{eq.transverse.ts} and by $S^*_GM := S^*M \cap
  T^*_GM $ the set of unit covectors in the $G$-transverse cotangent
  space $T^*_GM$. 

The previous set leads to the definition 
of $G$-transversally elliptic operators 
\cite{atiyahGelliptic,singer73}. 
Recall that a $G$-transversally
elliptic pseudodifferential operator on $M$ is a $G$-invariant pseudodifferential
operator whose principal symbol becomes invertible when restricted
to $T^*_GM\smallsetminus \{0\}$. Since $M$ is compact, we know
that this operators are generally not Fredholm due to the lack of full ellipticity. Nevertheless,
 the, now well known, Atiyah-Singer's result
 states that if $P$ is $G$-transversally elliptic then $\pi_\alpha(P)$
is Fredholm for any $\alpha \in \widehat{G}$, \cite{atiyahGelliptic,singer73}. 
This allows directly to define an index for $G$-transversally elliptic operators 
as an element of the $K$-homology of $C^*G$, the group $C^*$-algebra of $G$.
Furthermore, with little more work, Atiyah and Singer showed that this index is, in fact,
a $Ad$-invariant distribution on $G$. 
See also 
\cite{baldare:KK,Benameur:Baldare,Hochs,Julg,Kasparov:KKindex}
for related results and \cite{baldare:H,
BV2,PVTrans}
for index theorems on $G$-transversally elliptic operators using equivariant cohomology. 
The Fredholm property of this restrictions to isotypical component 
was the starting point for the study carried out in \cite{BCN}.

We now proceed to state the main result studied in this paper but first 
we need few more notations and definitions from \cite{BCLN1,BCLN2,BCN}.

Assume $M/G$ connected and let $K$ be a minimal isotropy subgroup of $G$, see \cite{Bredon,tomDieckTransBook}.
We shall say that $P$ is transversally $\alpha$-elliptic if for all $\xi \in (S^*_GM)^K$ the linear map
 \begin{equation*}
 \sigma_m(P)(\xi) \otimes \oid_{\alpha^*} : (E_{0\xi} \otimes \alpha^*)^K \rightarrow (E_{1\xi} \otimes \alpha^*)^K
 \end{equation*}
is invertible.

One of the main results of \cite{BCN} states that
$P$ is transversally $\alpha$-elliptic if, and only if, $\pi_\alpha(P)$ is Fredholm.
Here, we point out that the  transversal $\one$-ellipticity is related with transversal ellipticity on (singular) 
foliations 
\cite{AnSk11a, ConnesBook}.
%
For $G$ finite, this results were proved before
\cite{BCLN1, BCLN2}. 



In the present paper, we recall, in Definition \ref{def.alpha.invertible}, 
the notion of locally $\alpha$-invertible operator at $x\in M$ introduced in \cite{BCLN2} 
and we show in full generality the following result, see Theorem \ref{thm.main.Simonenko}. 

\begin{theorem*}
Assume that $M$ is a closed, smooth manifold and that $G$ is a compact Lie 
group acting smoothly on $M$.  
Let 
$P\in\psi^m(M;E_0,E_1)^G$ and $\alpha \in
\widehat{G}$. Then the following are equivalent:
\begin{enumerate}
  \item $\pi_\alpha(P) : H^s(M;E_0)_\alpha \to H^{s-m}(M;E_1)_\alpha$ is
    Fredholm for any $s \in \RR$,
  \item $P$ is transversally $\alpha$-elliptic,
  \item $P$ is locally $\alpha$-invertible.
\end{enumerate}
\end{theorem*}

Notice that in \cite{BCLN2}, the equivalence between $(1)$ and $(2)$ was stated without proof under the hypothesis that $\dim G< \dim M$. 
Moreover, the triple equivalence was then deduced in the case of finite group using the main result of \cite{BCLN2}. 
Here care is taken to state it in full generality and relax the hypothesis $\dim G < \dim M$. 
This proposition enlightens the results from \cite{BCLN1,BCLN2,BCN} in the sense that it explains the local computations done.

The previous theorem, as well as intermediate results in this paper, 
were obtained during discussions with R. Côme, M. Lesch and V. Nistor. 


 We point out that the Fredholm conditions obtained  in this paper  
are closely related to the ones in
\cite{SavinSchrohe}, for $G$-operators, and the ones in \cite{BL92},
for complexes of operators. 
Fredholm conditions were also investigated in different forms in 
\cite{BoutetdeMonvel,Schrohe99,schulze2001algebra,sternin1998general} for boundary problems and in  
\cite{chandler,favard,Ge,hagger2019limit,LindnerSeidel,Ma2,MaNi,RRS}
 using  techniques of limit operators and also $C^*$-algebras methods. 
 The techniques of limit operators are similar to the one used in \cite{BCN} 
 to obtain the Fredholm criterion for $\alpha$-transversally elliptic operator, 
 see also Section \ref{ssec.alpha.symbol}. 
Recent developments on singular operators including groupoid and $C^*$-algebras were 
accomplished in 
\cite{Yuncken,bohlen1,CarrilloRouseLescure1,
CarrilloRouseSo,CCQ,CNQ,Remi,DS1,LMN,LescureManchonVassout}.
 \\

\medskip

\noindent
{\em{Acknowledgements.}} I would like to express my gratitude to V. Nistor for useful discussions,
suggestions and encouragement during the redaction of this paper.
I would also like to thank M.-T. Benameur, R. Côme, P. Carrillo-Rouse, M. Lesch, P.-E. Paradan, M. Puschnigg and E. Schrohe 
for many helpful discussions. 
Last but not least, I would like to thank the referees for several very useful comments, suggestions and references.
This study was partially supported by DFG, project
SCHR 319/10-1.

\section{Preliminaries}
\label{sec.Preliminaries}

This section is devoted to background material and results. The reader can find
more details in \cite{BCLN1, BCLN2, BCN}. 
The reader familiar with
\cite{BCLN1,BCLN2,BCN} can skip this section at a first reading.

\subsection{Group actions}
\label{ssec.group.r}

Throughout the paper, we let $G$ be a \emph{compact} Lie group.
Assume that $G$ acts on a space $X$ and that $x \in X$, then $G x$ is the $G$ orbit of $x$ and
\begin{equation*}
   G_x := \{ g \in G \, \vert \ g x = x \} \subset
   G
\end{equation*}
the isotropy group of the action at $x$.

If $H \subset G$ is a subgroup, then
$X_{(H)}$ will denote the set of elements of $X$ whose isotropy
$G_{x}$ is conjugated to $H$ in $G$
and $G\times_{H} X$ the space
\begin{equation*}
G \times_H X  \ede (G \times  X)/\sim, 
\end{equation*}
where $(g h,x)\sim (g ,hx)$ for all $g \in G,
h\in H$, and $x\in X$.  

Let $V$ and $W$ be locally convex spaces
and $\maL(V; W)$ be the set of continuous, linear maps $V \to W$.
We let $\maL(V) := \maL(V; V)$. 
A representation of $G$ on $V$ is a continuous group morphism $G \to \maL(V)$,
where $\maL(V)$ is equipped with the strong operator topology. Said differently, the
map $G \times V \rightarrow V$ given by $(g, v) \mapsto gv$ is continuous and $v \mapsto gv$ is linear. 
We shall also call $V$ a $G$-module.
 
For any two $G$-modules $\maH$ and $\maH'$, we shall denote by
\begin{equation*}
    \Hom_{G}(\maH, \maH') \seq \Hom(\maH, \maH')^G \seq
    \maL(\maH, \maH')^G,
\end{equation*}
the set of continuous linear maps $T : \maH \to \maH'$ that commute
with the action of $G$, that is, $T (g v) = g T(v)$
for all $v \in \maH$ and $g\in G$.

Let $\maH$ be a $G$-module and $\alpha$ an irreducible
representation of $G$. Then $p_\alpha$ will denote the $G$-invariant
projection onto the {\em $\alpha$-isotypical component} $\maH_\alpha$ of
$\maH$, defined as the largest (closed) $G$-submodule of $\maH$
that is isomorphic to a multiple of $\alpha$. In other words,
$\maH_\alpha$ is the sum of all $G$-submodules of $\maH$ that are
isomorphic to $\alpha$. Notice that $\maH_{\alpha}
\simeq \alpha \otimes \Hom_{G}(\alpha, \maH)$ and 
\begin{equation*}\label{eq.isotypical}
  \maH_\alpha \, \neq \, 0 \ \Leftrightarrow \ \Hom_G(\alpha,
  \maH) \, \neq \, 0 \ \Leftrightarrow \ \Hom_G(\maH, \alpha) \,
  \neq \, 0 .
\end{equation*}
Recall that we denote by $\widehat{G}$ the set of equivalence classes of irreducible unitary representations of $G$.
Let $\chi_\alpha$ be the character of $\alpha \in \widehat{G}$ 
and  $z_\alpha:=\dim \alpha \chi_\alpha \in C^*G$ be the central projection 
associated in the group $C^*$-algebra $C^*G$ of $G$. 
Then $p_\alpha$ is the image of $z_\alpha$ induced by the group action on $\maH$.
If $T \in \maL(\maH)^G$
then
$T(\maH_\alpha) \subset \maH_\alpha$ and we let
\begin{equation*}
    \pi_\alpha : \maL(\maH)^G \to \maL(\maH_\alpha) \,, \quad
    \pi_\alpha(T) \ede  p_\alpha  T\vert_{\maH_\alpha}\,,
\end{equation*}
be the associated morphism. The morphism $\pi_\alpha$ will play an essential
role in what follows.

As before, we consider a compact Lie group $G$ and we now assume that $G$ acts by isometries on
a closed Riemannian manifold $M$.
Let $TM$ and $T^*M$ be respectively the {\em tangent} and {\em cotangent bundle} on $M$ 
and recall that they can be identified using the $G$-invariant Riemannian metric on $M$.
Let $S^*M$ denote the {\em unit cosphere bundle} of 
$M$, that is, the set of unit vectors in $T^*M$, as usual. 
Denote by $\mathfrak{g}$ the Lie algebra of
$G$. Then any $Y \in \mathfrak{g}$ defines as usual the vector field
$Y_M$ given by $Y_M(m)=\frac{d}{dt}_{|_{t=0}} e^{tY}\cdot
m$. Denote by $\pi : T^*M \rightarrow M$ the canonical projection and
let us introduce as in \cite{atiyahGelliptic} the {\em $G$-transversal
space}
\begin{equation}\label{eq.transverse.ts}
    T^*_G M \ede \{\xi \in T^*M\ |\ \xi(Y_M(\pi(\xi)))=0,
    \forall Y\in \mathfrak{g}\}.
\end{equation}

We denote by $T_GM$ the image of $T^*_GM$ in $TM$ obtained using the Riemannian metric. 
In other words, $T_GM$ is the orthogonal to the orbits in $TM$.  
Finally, let $S^*_GM$ be the set of unit covectors in $T^*_GM$, that is $S^*_GM=S^*M \cap T^*_GM$.

%
Recall that if 
$M/G$ is connected, there is a minimal isotropy subgroup $K$ such that any isotropy subgroup of $G$ acting on $M$ contains
a subgroup conjugated to $K$ and $M_{(K)}$ is an open dense submanifold of $M$ called the principal orbit bundle, see 
\cite[Section IV. 3]{Bredon} and \cite[Section I. 5]{tomDieckTransBook}.

\subsection{Pseudodifferential operators}
\label{sub.psdo}
Let $G$ be a compact Lie group acting smoothly by isometries on a compact, Riemannian manifold without boundary $M$ as before.
We shall denote by $\op{m}$ the space of
order $m$, {\em classical} pseudodifferential operators on $M$.
Let $\clop$ and $\clopn$ denote the respective norm closures of $\op{0}$ 
and $\op{-1}$. The action of $G$ then extends to a continuous
action on $\op{m}$, $\clop$, and $\clopn$, see \cite{AS1} for example. 
We shall denote by
$\maK(\maH)$ the algebra of compact operators acting on a Hilbert
space $\maH$. Of course, we have $\clopn = \maK(L^2(M; E))$.

 We shall
denote, as usual, by $\maC(S^*M; \End(E))$ the set of continuous
sections of the {\em lift} of the vector bundle $\End(E) \to M$ to
$S^*M$.
We have the following well known exact sequence
\begin{equation*}\label{cor.full.symbol}
   0 \, \to \, \maK(L^2(M; E))^G \, \to \, \clop^G\,
   \stackrel{\sigma_{0}}{-\!\!\!\longrightarrow}\,
   \maC(S^*M; \End(E))^G \, \to \, 0 \,.
\end{equation*}
See, for instance, \cite[Corollary 2.7]{BCLN1},
where references are given.

Recall that a $G$-invariant classical pseudodifferential operator $P$
of order $m$ is said {\em elliptic} if its principal
symbol is invertible on $T^* M\smallsetminus \{0\}$
and  {\em $G$-transversally elliptic} if its principal
symbol is invertible on $T^*_G M\smallsetminus \{0\}$
\cite{atiyahGelliptic,AS1, PVTrans}, see Equation \eqref{eq.transverse.ts} for the definition of $T^*_GM$.

We may now state the classical result of
Atiyah and Singer \cite[Corollary 2.5]{atiyahGelliptic}. 
\begin{theorem}[Atiyah-Singer \cite{atiyahGelliptic,singer73}]
\label{Thm.atiyahGell}
Let $P$ be a $G$-transversally elliptic operator. Then, for every irreducible
representation $\alpha\in \widehat{G}$, $\pi_\alpha(P) : H^s(M;
E_0)_\alpha \, \to \, H^{s-m}(M; E_1)_\alpha,$ is Fredholm.
\end{theorem}

Let us recall the following fact which is a direct consequence 
of the fact that $G$ acts by unitary multiplier on $\maK(\maH)$.  

\begin{proposition} \label{prop.image}
We have natural isomorphisms
\begin{align*}
  p_\alpha\overline{\psi^{-1}}(M;E)^G  & \simeq
  \pi_\alpha(\overline{\psi^{-1}}(M;E)^G) 
  & = \pi_\alpha(\maK(L^2(M; E))^G) =
  \maK(L^2(M;E)_\alpha)^G\,,
\end{align*}
where the first isomorphism map is simply $\pi_\alpha$ and
\begin{equation*}
  \maK(L^2(M; E))^G = \clopn^G \simeq \oplus_{\alpha \in
    \widehat{G}}\maK(L^2(M; E)_\alpha)^G\,.
\end{equation*}
\end{proposition}

\begin{proof}
See, for example, \cite[Section 3]{BCLN1} for a proof.
\end{proof}

\subsection{$\alpha$-transversally elliptic operators}
\label{ssec.alpha.symbol}
Let $G$ be a compact Lie group acting smoothly by isometries on a compact, Riemannian manifold without boundary $M$ as before.
Let $P : \maC^\infty(M,E_0) \rightarrow \maC^\infty(M,E_1)$ be a $G$-invariant pseudodifferential operator.
Let $p : M \rightarrow M/G$ be the projection.
Let  $ M/G =\bigsqcup_{i\in I} C_i$ be the decomposition into connected components of $M/G$.
Notice that $I$ is finite and let $K_i \subset G$ be a minimal isotropy group for $M_i:=p^{-1}(C_i)$. 
Denote by $(S^*_GM_i)^{K_i}$ the subset of $K_i$-invariant elements of $S^*_GM_i$, see Equation \eqref{eq.transverse.ts}.

\begin{definition}\cite{BCN}\label{def.chi.ps}
We shall say that $P \in \psi^m(M; E_0, E_1)^{G}$ is {\em transversally 
  $\alpha$-elliptic} if for any $i\in I$, and $\xi \in (S^*_GM_i)^{K_i}$,
  \begin{equation*}
  \sigma_m(P)(\xi) \otimes \oid_{\alpha^*} : (E_{0\xi} \otimes \alpha^* )^{K_i} \rightarrow (E_{1\xi} \otimes \alpha^* )^{K_i}
  \end{equation*}
   is invertible. 
\end{definition}

Let us recall the main result of \cite{BCN}, see also \cite{BCLN1,BCLN2} for finite groups.

\begin{theorem}\label{thm.main1}\cite{BCN}
Let $m\in \RR$, $P \in \psi^m(M; E_0, E_1)^{G}$ and $\alpha \in \widehat{G}$. Then
\begin{equation*}
    \pi_\alpha(P) : H^s(M; E_0)_\alpha \, \to \, H^{s-m}(M;
    E_1)_\alpha
\end{equation*}
is Fredholm if, and only if $P$ is transversally $\alpha$-elliptic.
\end{theorem}

We now briefly relate Definition \ref{def.chi.ps} with the notion of limit operators, 
see \cite{chandler,favard,Ge,hagger2019limit,LindnerSeidel,Ma2,MaNi,RRS}.
In order to simplify notations, let us assume $\alpha=\one$ 
and $M/G$ connected and let $K$ be the minimal isotropy subgroup.
We follow \cite{BCN}.
Let $(x_0 ,\xi) \in S^*_GM_{(K)}$ and assume $G_{x_0}=K$. 
Let $U \subset (T_GM)_{x_0}  \simeq (T^*_GM)_{x_0}$ be a slice 
at $x_0$, let $W = G\exp_{x_0}(U) \cong G/K \times U$ be the
associated tube around $x_0$, and let
\begin{equation*}
     \eta \in E_{x_0}^{K}\ \mbox{ and } \ f \in \maC^{\infty}_c(U)\,, \ f(x_0) = 1 \,.
\end{equation*}
Notice that $(S^*_{K} U)_{x_0} = S^*_{x_0} U$,
because $x_0 \in M_{(K)}$ and hence $\xi \in S^*_{x_0} U$.
Let us define $s_\eta \in \maC^{\infty}_c(W; E)^G$ and $e_t \in \maC^\infty(W)^G$ by
$s_\eta( g \exp_{x_0} (y) ) \ede f(y) g\eta$ and 
$e_t( g \exp_{x_0} (y)) \ede e^{\imath t \langle y, \xi \rangle}$, $t \in \RR$.
In other words, they are the functions on $W$
extending the functions $y \mapsto f(y) \eta$ and $y \mapsto
e^{\imath t \langle y, \xi \rangle}$ defined on $U \subset T_{x_0}U 
= (T_KU)_{x_0}$ by $G$-invariance via $W = G
\exp_{x_0}(U)$.
Using oscillatory testing techniques, see, for instance \cite{Hormander3, Treves}, 
the following proposition can be shown, see \cite{BCN}.

\begin{proposition}\label{prop.injective}
Assume that $0 \neq \eta \in E_{x_0}^K$. Then, for 
every $P \in \psi^0(M; E)$, we have $\lim_{t \to \infty } P ( e_t s_\eta ) (x_0) \seq
  \sigma_0(P)(\xi)\eta$.
In particular, if $P\in\psi^0(M; E)^G$, then
\begin{equation*}
  \lim_{t \to \infty } \pi_{\one}(P) ( e_t s_\eta ) (x_0) \seq
  \sigma_0(P)(\xi)\eta =: \, \pi_{(\xi, \one)} \big (\sigma_0(P)\big )
  \eta\,.
\end{equation*}
\end{proposition}

\begin{remark}
Let $t>0$ and $V_0=\oid : E_{x_0}^K \rightarrow E_{x_0}^K$.
Let $V_t : E_{x_0}^K \rightarrow C^\infty(M,E)^G$ be the map given by $V_t(\eta)=e_t s_\eta$ 
and let $V_{-t}=\operatorname{ev}_{x_0} : C^\infty(M,E)^G \rightarrow E_{x_0}^K$ be the evaluation map at $x_0$.
Then we have $V_{-t}V_t = V_0 = \oid_{E_{x_0}^K}$ and
\begin{equation}\label{eq.limit.op}
\sigma_0(P)(\xi)=\lim \limits_{t\to \infty} V_{-t}\pi_\one(P)V_t : E^K_{x_0} \rightarrow E^K_{x_0}.
\end{equation}
Equation \eqref{eq.limit.op} is similar to the definition of limit operators, 
see \cite{chandler,favard,Ge,hagger2019limit,LindnerSeidel,Ma2,MaNi,RRS}.
\end{remark}

\section{Simonenko's general localization principle}
\label{ssub.Simonenko}
In this section, we recall the essentials of the usual Simonenko's localization
principle \cite{simonenko1965new}, see also \cite{SimonenkoNgok}. 
The results of this section are well-know from experts, we shall include proofs for the convenience of the reader.
We refer in particular to \cite[Chapter 2]{RRS} and \cite[Chapter 2]{LocalPrincipleBook}, where more general situations
are treated. 
 The general localization principle of this section will be used in the sequel
to deduce Fredholm conditions for restriction to isotypical components of invariant operators 
on closed manifolds.

Throughout this section, we let $T$ be a compact Hausdorff topological space 
and $\maC(T)$ be the $C^*$-algebra of complex valued continuous functions on $T$.
Let $A$ be a unital $C^*$-algebra and assume that $\maC(T)$ identifies with
 a unital sub-$C^*$-algebra in $A$, meaning, in particular,
that the image of the unit $1_Z$ of $\maC(T)$ is the unit $1_A$ of $A$. 

\begin{definition}\label{def.ssp}
An element $a \in A$ is said to {\em have the strong Simonenko local property} with
respect to $\maC(T)$ if, for every $\phi, \psi \in \maC(T)$ with compact disjoint
supports, we have $\phi a \psi = 0$.
\end{definition}
 
The following lemma follows for example from similar arguments 
as in \cite[Theorem 2.1.6]{RRS} and \cite[Theorem 2.5.6]{LocalPrincipleBook}.

\begin{lemma}\label{lemma:commute} 
The set $B \subset A$ of elements $a\in A$ satisfying the strong Simonenko
local property is the set of elements of $A$ commuting with $\maC(T)$.
\end{lemma}

\begin{proof}
We are going to show that the set of elements $a \in A$ with the
strong Simonenko local property is a $C^*$-algebra $B$ containing
$\maC(T)$ and that every irreducible representation of $B$ restricts to a
scalar valued representation on $\maC(T)$, and hence that $\maC(T)$
commutes with $B$.

Let us show first that $B$ is a sub-$C^*$-algebra of $A$. Note that
$B$ is not empty since $\maC(T) \subset B$. To show that $B$ is a
sub-$C^*$-algebra, the only fact that is non-trivial to prove is that
$ab \in B$, for all $a,b \in B$. Let $\phi $ and $\psi \in \maC(T) $ with
disjoint compact supports and let $\theta$ be a function equal to $1$
on $\mathrm{supp}(\psi)$ and $0$ on $\mathrm{supp}(\phi)$, which
exists by Urysohn's lemma. Then we have
\begin{equation*}
  \phi ab\psi = \phi a(\theta +1-\theta)b \psi =\phi a \theta b\psi
  +\phi a(1-\theta )b\psi = 0,
\end{equation*}
since $\phi a \theta=0$ and $(1-\theta )b\psi=0$.

Let $\pi : B \rightarrow \mathcal{L}(H)$ be an irreducible
representation of $B$.  First, let us show that for any $\phi, \psi
\in \maC(T)$ with disjoint support, we either have $\pi(\phi)=0$ or
$\pi(\psi) =0$. Indeed we have $\pi(\phi)\pi(a)\pi(\psi)=0$ since
$\phi a \psi =0$, for any $a \in B$. Assume that $\pi(\psi) \neq 0$
then there is $\eta \in H$ such that $\pi(\psi)\eta \neq 0$. Now,
$\pi$ is irreducible so we get that the set $\{\pi(a)\pi(\psi)\eta
,\ a \in B \}$ is dense in $H$. Thus $\pi(\phi)=0$ on a dense subspace
of $H$ and so on $H$.

Assume now that $\pi(\maC(T)) \neq \CC 1_H$. Then there exist two distinct
characters $\chi_0, \chi_1 \in \Sp(\pi(\maC(T)))$. Denote by $h_\pi :
\Sp(\pi(\maC(T))) \to \Sp(\maC(T))=T$ the injective map adjoint to $\pi$, and choose
$\phi, \psi \in \maC(T)$ with disjoint supports such that
$\phi(h_\pi(\chi_0)) = 1$ and $\psi(h_\pi(\chi_1)) = 1$. Then
$\pi(\phi)(\chi_0) = 1$ and $\pi(\psi)(\chi_1) = 1$, which contradicts
the fact that either $\pi(\phi) = 0$ or $\pi(\psi) = 0$.
\end{proof}


We now fix notations and hypothesis that will remain valid until the end of this section.

\begin{NotHyp}\label{NotHyp}
As before, let $T$ be a compact Hausdorff topological space and denote by $\maC(T)$ the $C^*$-algebra of continuous functions on $T$.
Let $\maG$ be a Hilbert space 
and let $ \maC(T) \rightarrow \maL(\maG)$ be a non degenerate faithful representation 
({\em i.e.} $\maC(T)$ identifies with its image in $\maL(\maG)$ and
the image of the constant function $1$ is $\oid\in \maL(\maG)$). 
Assume that the image of $\maC(T)$ does not intersect $\maK(\maG)\smallsetminus \{0\}$. 
In other words, we are assuming that $\maC(T)$ identifies with a unital sub-$C^*$-algebra of
the Calkin algebra $\maQ(\maG):=  \maL(\maG)/\maK(\maG)$.
We shall denote by $M_\phi$ the image of a function $\phi\in \maC(T)$ 
in $\maL(\maG)$ and call it the multiplication operator by $\phi$.
\end{NotHyp}

\begin{remark}\label{rem.positive.measure}
If $X$ is a locally compact space and 
$\maG=L^2(X,\mu)$ then the representation of $\maC_0(X)$ is faithful 
if and only if $\mu$ is a strictly positive
measure, {\em i.e.} $\mu(U)>0$ for every open set $U \subset X$. In this case,
the only compact operator in $\maC_b(X)$ is zero, where $\maC_b(X)$ 
denotes the $C^*$-algebra of bounded continuous function, 
see Lemma \ref{lemma.cap.compact} below for more details. 
\end{remark}

We shall now turn to the definition of local invertibility.
The definition in the present paper and in for example \cite[Section 2.5]{LocalPrincipleBook}
are at the first reading not the same 
but they describe the same property by Lemma \ref{lemma:commute}. See also
\cite[Section 2.4.1]{LocalPrincipleBook}.

\begin{definition}\label{def.loc.inv}
An operator $P \in \maL(\maG)$ is said to be
{\em locally invertible} at $x\in T$ if there exist:
\begin{enumerate}[(i)]
  \item a neighbourhood $V_x$ of $x$ and 
  \item operators $Q_1^x$ and $Q_2^x \in \maL(\maG)$
\end{enumerate}
such that, for all $\phi \in \maC_c(V_x)$
\begin{equation*}
  Q_1^xPM_\phi \seq M_\phi \seq  M_\phi PQ_2^x \in \maL(\maG).
\end{equation*}
The operator $P$ is said to be {\em locally invertible on $T$} if it is
locally invertible at any $x\in T$.
\end{definition}

\begin{notation}\label{not.Simonenko}
We let $\Psi_T(\maG) \subset \maL(\maG)$ denote the $C^*$-algebra 
consisting of all $P \in \maL(\maG)$ such that $M_\phi P M_\psi \in 
\maK(\maG)$, for all $\phi, \psi \in \maC(T)$ with disjoint support. 
We let $\maB_T(\maG)$ denote the image of $\Psi_T(\maG)$ in the Calkin algebra 
$\maQ(\maG) $.\\
In other words, $\maB_T(\maG)=q(\Psi_T(\maG))$ where $q : \maL(\maG) \rightarrow \maQ(\maG)$ is the canonical projection.
\end{notation}

\begin{remark}\label{rem.PSI.T}
We know by Lemma \ref{lemma:commute} that
\begin{equation*}
  \maB_T(\maG)=\{P \in \mathcal{Q}(\maG) \mid \text{$M_\phi P =PM_\phi$ for all
  $\phi \in \maC(T)$}\}.
\end{equation*}
Said differently, $\Psi_T(\maG)$ is the essential commutant of $\maC(T)$, that is
\begin{equation*}
\Psi_T(\maG)=\{P \in \maL(\maG),\ M_\phi P - PM_\phi \in \maK(\maG),\ \forall \phi \in \maC(T)\},
\end{equation*}
see the relation with the work \cite{xia}. 
Moreover, the family of morphisms 
\begin{equation*}
\maB_T(\maG) \rightarrow \maB_T(\maG)/\ker(\operatorname{ev}_x)\maB_T(\maG),
\end{equation*} 
$x \in T$ is exhaustive, see \cite[Definition 3.1]{nistor2014exhausting} for the precise definition.
This follows from the definition of the central character map, see for example \cite[Remark 2.11]{BCLN2}. 
\end{remark}

I would like to thank an anonymous referee for pointing out to me 
how to simplify the proof of the next proposition and for the reference
 \cite[Proposition 2.2.3]{RRS} where a more general situation is treated.

\begin{proposition}\label{cor.loc.inv.implies.Fred}
Assume that $P \in \Psi_T(\maG)$ is locally invertible on $T$. Then $P$ is Fredholm.
\end{proposition}

\begin{proof}
By assumption $P$ is locally invertible on $T$ therefore for any $x \in T$ there are open neighborhood $V_x$ 
and operators $Q_1^{x}$, $Q_2^x$ such that 
for all $\phi \in \maC_c(V_x)$,
\begin{equation*}
  Q_1^xPM_\phi \seq M_\phi \seq  M_\phi PQ_2^x \in \maL(\maG).
\end{equation*}
Since $T$ is compact, there are $x_1, \cdots , x_N$ such that $(V_{x_j})_{j=1}^N$ is a finite open cover of $T$.
Now let $(\phi_j)_{j=1}^{N}$ be a partition of unity subordinated to $(V_j)_{j=1}^N$ then
for all $j=1,\cdots ,N$, we have
\begin{equation*}
  Q_1^{x_j}PM_{\phi_j} \seq M_{\phi_j} \seq  M_{\phi_j} PQ_2^{x_j} \in \maL(\maG).
\end{equation*}  
It follows that $Q^1:= \sum_{j=1}^N  Q_1^{x_j} M_{\phi_j}$ and $Q^2 := \sum_{j=1}^N  M_{\phi_j} Q_2^{x_j} $
are respectively left inverse and right inverse of $P$ modulo compact operators. 
Indeed, if $[A,B]=AB-BA$ denotes the commutator, we have
\begin{align*}
Q_1 P &= \sum_{j=1}^N  Q_1^{x_j} M_{\phi_j} P= \sum_{j=1}^N  Q_1^{x_j} [M_{\phi_j} , P] + \sum_{j=1}^N  Q_1^{x_j} P M_{\phi_j}\\
&=\sum_{j=1}^N  Q_1^{x_j} [M_{\phi_j} , P] + \sum_{j=1}^N  M_{\phi_j}\\
&=\sum_{j=1}^N  Q_1^{x_j} [M_{\phi_j} , P] + \oid,
\end{align*} 
and similarly
\begin{equation*}
PQ_2 = \sum_{j=1}^N [P,M_{\phi_j}]Q_2^{x_j} +\oid.
\end{equation*}
Since $P \in \Psi_T(\maG)$, we know from Remark \ref{rem.PSI.T} that  $[P,M_{\phi_j}]$ is compact. 
Thus, $\sum_{j=1}^N  Q_1^{x_j} [M_{\phi_j} , P]$ and $\sum_{j=1}^N [P,M_{\phi_j}]Q_2^{x_j}$ are compact operators and
therefore $P$ is Fredholm 
%

\end{proof}

%


\begin{definition}\label{def.property.strong.CV}
We shall say that the representation $\maC(T) \rightarrow \maL(\maG)$ of Notations and Hypothesis \ref{NotHyp} 
has the property of strong convergence to $0$ if for any 
$x\in T$ 
$M_{\chi_V}$ converges strongly to zero, where $V$ runs the set of neighborhoods of $x$ 
and $\chi_V\in \maC(T,[0,1])$ is equal to $1$ on a neighborhood of $x$, with values in $[0,1]$
 and is supported in $V$.
Said differently, $\maC(T) \rightarrow \maL(\maG)$ as the property of strong convergence to $0$ if  
$\forall x\in T$, $\forall h\in \maG$, $\forall \varepsilon >0$, there is a neighborhood $V'$ of $x$
such that for any neighborhood $V$ of $x$, if $V\subset V'$ then $\|M_{\chi_V}h\|<\varepsilon$.
\end{definition}

\begin{proposition}[General Simonenko's localization principle]
\label{prop.simonenko.gen}
Let $P \in \Psi_T(\maG)$. 
Assume that $\maC(T) \rightarrow \maL(\maG)$ is as in Notation and Hypothesis \ref{NotHyp} 
and has the property of strong convergence to $0$, see Definition \ref{def.property.strong.CV}.
Then $P$ is locally invertible on $T$ if, and
only if, $P$ is Fredholm.
\end{proposition}

\begin{proof}
The first implication is exactly Proposition \ref{cor.loc.inv.implies.Fred}.\\
Let us prove the opposite implication.
That is, let us assume that $P$ is Fredholm and 
let us prove that $P$ is locally invertible at $x \in T$,
where $x$ is fixed, but arbitrary. 
To this end, let $Q \in \maL(\maG)$ be an inverse
modulo $\mathcal{K}(\maG)$ for $P$,
i.e.\ $PQ=\oid+K$ and $QP=\oid+K^\prime$, with
$K,K' \in \maK(\maG)$. Using Lemma 
\cite[Proposition 1.3.10]{Dixmier}, we can assume that
  $Q\in \Psi_T(\maG)$ if one desires. Let $\chi \in \maC(T)$ be equal
to $1$ on a neighbourhood $V_x$ of $x$, with values in $[0,1]$ and supported in a neighborhood $V_x'$.
Let $\phi
\in \maC_c(V_x)$ then
\begin{equation*}
\begin{gathered}
  M_\phi M_\chi PQM_\chi \seq M_\phi M_\chi^2+M_\phi M_\chi KM_\chi \qquad
\mbox{and} \\
 M_\chi Q P M_\chi M_\phi \seq M_\chi^2M_\phi + M_\chi
K^\prime M_\chi M_\phi \,.
\end{gathered}
\end{equation*}
Since $\phi$ is supported in $V_x$, we have
$\phi\chi=\phi$ and so
\begin{equation*}
  M_\phi PQ M_\chi  \seq M_\phi (1+  M_\chi K M_\chi)\quad
\mbox{and} \quad  M_\chi Q P M_\phi \seq (1+  M_\chi
K^\prime M_\chi)M_\phi \,.
\end{equation*}

As $V_x'$ becomes small, we have that $ M_\chi$
  converges strongly to $0$ because $\maC(T) \rightarrow \maL(\maG)$ 
  has the property of strong convergence to $0$, see Definition \ref{def.property.strong.CV}. 
  Since $K$ is compact, we obtain that
  $\| M_\chi K  M_\chi \| \to 0$. Thus, by choosing $V_x'$ small enough, we
  may assume that $\| M_\chi K M_\chi\| < 1$ and $\| M_\chi K'  M_\chi\| < 1$.

It follows that $(1+  M_\chi K M_\chi)$ and $(1+  M_\chi K^\prime M_\chi)$ are
invertible and this implies
\begin{equation*}
\begin{gathered}  
   M_\phi P\big(Q M_\chi(1+  M_\chi K M_\chi)^{-1}\big)
   \seq M_\phi\qquad \mbox{and}\\
    \big((1+  M_\chi K^\prime M_\chi)^{-1} M_\chi
   Q\big)PM_\phi \seq M_\phi \,,
   \end{gathered}
\end{equation*}
that is, $P$ is locally invertible. This completes the second
implication, and hence the proof.
\end{proof}



Simonenko's principle is then \cite{simonenko1965new}:

\begin{proposition}[Simonenko's principle] \label{prop.simonenko}
Let $M$ be a closed manifold, $E$ a hermitian vector bundle on $M$,
$\Psi_M(L^2(M,E))$ be as in \ref{not.Simonenko} and let
$P \in \Psi_M(L^2(M,E))$. We have that $P$ is locally invertible on $M$ if, 
and only if, it is Fredholm.
\end{proposition}

\begin{proof}
The hypothesis of  Proposition \ref{prop.simonenko.gen} are satisfied because if $h\in L^2(M,E)$ 
then $\int_{V} |h|^2 \mathrm{dvol}$ goes to $0$ when the volume of $V$ goes to $0$, see also Remark \ref{rem.positive.measure}. 
\end{proof}


\section{Equivariant local principle for closed manifolds}\label{sec.Simonenko.closed}
Let $G$ be a compact Lie group that we assume to act smoothly by isometries on a closed Riemannian manifold
$M$ as before. We shall denote, as before, by $\widehat{G}$ the set 
of isomorphism classes of irreducible unitary representations of $G$. 
Let $\mathcal{H} := L^{2}(M,E)$ and
let $\maH_\alpha \cong \alpha \otimes (\alpha^* \otimes \maH)^G$ be the $\alpha$-isotypical component associated to
$\alpha \in \widehat{G}$, as in the introduction and section \ref{sec.Preliminaries}.

Any $\phi \in \maC(M)^G$ acts
by multiplication on $\maH_\alpha$ and we shall
denote also by $M_\phi$ the induced multiplication operator, as in section \ref{ssub.Simonenko}.
Furthermore, the representation of $\maC(M/G)=\maC(M)^G$ given by the previous multiplication operator
on $\maH$ and $\maH_\alpha$ are non degenerate.

\begin{definition}\label{def.alpha.invertible}
We shall say that $P \in \maL(\maH)^{G}$ is {\em locally $\alpha$-invertible at
  $x\in M$} if $\pi_\alpha(P)$ is locally invertible 
  at $Gx \in M/G$, see Definition \ref{def.loc.inv}.
  
\end{definition}

We let $\Psi_M^G(\maH)$ denote the $G$-invariant elements in the
$C^*$-algebra $\Psi_M(\maH)$, which was defined in \ref{not.Simonenko},
in the previous subsection.
More precisely, using Remark \ref{rem.PSI.T}
\begin{equation}\label{eq.PsiMG}
\Psi_M^G(\maH)=\{P\in \maL(\maH)^G\ |\ [P,M_\phi]\in \maK(\maH),\ \forall \phi \in \maC(M)\}.
\end{equation}

Before tackling the Simonenko's equivariant localization principle, let us first justify our hypothesis with the following simple example.
\begin{example}
Let $M=G$ be our manifold with its standard action by translation. In this case, $T^*_GM=G \times \{0\}$ and then every $G$-invariant pseudodifferential operator $P$ is $G$-transversally elliptic. It follows that the restriction $\pi_\alpha(P)$ to any isotypical component is Fredholm. Let us then consider the null operator $0 : L^2(G) \rightarrow 0$. Clearly, the restriction to the isotypical component associated with the trivial representation $\one=\CC$ of $G$ is Fredholm. In other words, $0=\pi_\one(0) : L^2(G)^G=\CC \to 0$ is Fredholm. But obviously, $0=\pi_\one(0)$ is not locally $\one$-invertible. 
\end{example}

This pathological example arises from the fact that there are points $x\in M$ such that the slice at $x$ are discrete 
(and in fact on the whole space $M=G$ in the previous example). Nevertheless, we can extract such a pathological points using the following interesting fact.
\begin{lemma}\label{lemma.finite.component}
Let $X$ be a not necessarily compact $G$-manifold without boundary and let $x \in X$ be such that $(T^*_GX)_x=\{0\}$. Then the orbit of $x$ is a union of connected components of $X$ in bijection with the connected component of $G/G_x$.
\end{lemma}
\begin{proof}
Since $(T^*_GX)_x=\{0\}$, we obtain that $S_x=\{x\}$ is the only slice at $x$. From the slice theorem, 
we deduce that the orbit $Gx\cong G\times_{G_x} \{x\}=G \times_{G_x} S_x$ is open but it is also 
compact. Therefore, $Gx$ is a union of connected components of $X$
 in bijection with the connected components of $G/G_x\cong Gx$.
\end{proof}

Consider the set of points 
\begin{equation}\label{eq.def.maP}
\maP:=\{x \in M, \ (T^*_GM)_x=\{0\}\}.
\end{equation} 
Then $M$ is the disjoint union of the closed manifolds $M\smallsetminus \maP$ and $\maP$. 
Indeed, $\maP$ is a union of clopen orbits and therefore it is also compact because $M$ is.
The same argument also implies that $M\smallsetminus \maP$ is a closed submanifold of $M$.

\begin{remark} The set $\maP$ will be empty for example in the following cases:
\begin{enumerate}
\item if $M$ is connected and not reduced to a single orbit,
\item if $\dim M > \dim G$,
\item in particular, if $M/G$ is an orbifold of dimension $>0$.
\end{enumerate}
\end{remark}
In other cases, we can use the following useful result.
\begin{lemma}\label{lemma.pathological:points}
Let $\maP=\{x \in M, \ (T^*_GM)_x=\{0\}\}$ be the clopen introduced in Equation \eqref{eq.def.maP}.
Let $\chi$ be the characteristic function of the clopen $M\smallsetminus \maP$. 
Let then $M_\chi$ be the multiplication operator by $\chi$. Let $P\in \Psi_M(\maH)$ then $P=P_1 + P_2 +P_3$ 
with $P_1=M_\chi P M_\chi \in \Psi_{M\smallsetminus \maP}(\maH)$, $P_2=M_{1-\chi} P M_{1-\chi}  \in \Psi_\maP(\maH)$ 
and $P_3=M_{\chi} P M_{1-\chi} +M_{1-\chi} P M_{\chi}   \in \maK(\maH)$. 
Furthermore, if $P\in \Psi_M^G(\maH)$ then $\pi_\alpha(P_2)$ is Fredholm for any $\alpha \in \widehat{G}$ 
and therefore $\pi_\alpha(P)$ is Fredohlm if, and only if, $\pi_\alpha(P_1)$ is.
\end{lemma}
\begin{proof}
The first part is clear since $M_\chi$ and $M_{1-\chi}$ have disjoint supports. 
For the second part, decompose $\maP=\sqcup_{i=1}^N \maP_i$ into clopen orbits 
and let $\phi_i$ be the corresponding characteristic functions 
then as before we can write $P_2=\sum \limits_{i=1}^N M_{\phi_i}P_2 M_{\phi_i} + C$, 
where $C = \sum_{i\neq j} M_{\phi_i} P_2 M_{\phi_j}$ is compact 
because the supports of $\phi_i$ and $\phi_j$ are disjoint for $i\neq j$. 
The previous decomposition is in fact a decomposition into $G$-invariant operators since $\phi_i$ is $G$-invariant. 
Therefore, for every $ \alpha \in \widehat{G}$, $\pi_\alpha(P_2)$ is Fredholm if, 
and only if, $\pi_\alpha(M_{\phi_i}P_2 M_{\phi_i})$ is Fredholm for any $i$. 
Now notice that $\maP_i \cong G \times_{G_x} \{x\}$ for some $x \in \maP$ therefore $\forall \alpha \in \widehat{G}$,
\begin{align*}
L^2(\maP_i,E\vert_{\maP_i})_\alpha &\cong \alpha \otimes \bigg(\alpha^* \otimes L^2(\maP_i,E\vert_{\maP_i})\bigg)^G\\
& \cong \alpha \otimes \bigg(\alpha^* \otimes L^2(G \times_{G_x} \{x\},G \times_{G_x} E_x)\bigg)^G
&\cong \alpha \otimes \bigg( \alpha^* \otimes E_x\bigg)^{G_x}
\end{align*}
is finite dimensional. \\
It follows that there are no condition for the restriction $\pi_\alpha(P_2)$ to be Fredholm, $\forall \alpha \in \widehat{G}$. 
In other words, for every $ \alpha \in \widehat{G}$, $\pi_\alpha(P)$ is Fredholm if, and only if, $\pi_\alpha(P_1)$ is.
\end{proof}

\begin{remark}
Notice that the previous proof implies that the image of $\maC(M)^G$ in $\maL(\maH_\alpha)$ 
intersects $\maK(\maH_\alpha)$ when $\maP \neq \emptyset$. 
\end{remark}

Recall that if $x \in M$ then $W_x \cong G \times_{G_x} U_x$ denotes a tube around $x$ 
and $U_x$ a slice at $x$, see \cite[Section I. 5]{tomDieckTransBook}.
Moreover, we have a $G$-equivariant isomorphism of vector bundles $E \cong G \times_{G_x} (U_x \times E_x)$. 
The next lemma could also be deduced from \cite[Corollary 1.5]{Bruning78}.

\begin{lemma}\label{lem.L2.tube.non.zero}
Let $\alpha \in \widehat{G}$ and 
let $\maH_\alpha=L^2(M,E)_\alpha\cong \alpha \otimes L^2(M,E\otimes \alpha^*)^G$.
The subset 
\begin{equation}\label{eq.lem.N}
\maN_\alpha :=\big\{x\in M, \ \exists W_x,\ \text{such that } L^2(W_x,E)_\alpha=\{0\}\big\}
\end{equation}
 is a $G$-invariant clopen. 
\end{lemma}

\begin{proof}
Replacing $E$ with $E\otimes \alpha$, we see that we can assume that 
$\alpha $ is the trivial representation and therefore that $L^2(W_x,E)_\alpha = L^2(W_x,E)^G$.
Clearly, $\maN_\alpha$ is $G$-invariant. We shall denote simply $\maN_\alpha$ by $\maN$ is this proof
since we consider the trivial representation.

Notice now that 
\begin{equation}\label{eq.N}
\maN=\big\{x\in M, \ \forall W_x,\ L^2(W_x,E)^G=\{0\}\big\}
\end{equation}
because if $W_x$ and $W_x'$ are two tubes around $x \in M$
then
\begin{align*}
L^2(W_x,E)^{G}&\cong L^2(U_x,E_x)^{G_x}\cong L^2(U_x',E_x)^{G_x}\cong L^2(W_x',E_x)^G.
\end{align*}

Let us show that $\maN$ is open. Let $x \in \maN$. By definition, there is $W_x$ such that
$L^2(W_x,E)^G =\{0\}$. Let $y \in W_x$ and assume that there is a tube $W_y$ around $y$
such that $L^2(W_y,E)^G\neq \{0\}$. By $G$-invariance of $W_x$ and $W_y$, we see that
we can assume $W_y$ small enough such that $W_y \subset W_x$.
But then $\{0\}\neq L^2(W_y,E)^G \subset L^2(W_x,E)^G$ which contradicts the fact that $x \in \maN$.
Therefore, $W_x \subset \maN$ and $\maN$ is open.

We now show that $\maN$ is closed. Let $x \in M\smallsetminus \maN$. 
By Equation \eqref{eq.N}, there is $W_x$ such that $L^2(W_x,E)^G=L^2(U_x,E_x)^{G_x}\neq \{0\}$.
Let $K$ be a minimal isotropy subgroup for the linear action of $G_x$ on $U_x$, see \cite[Section 5]{tomDieckTransBook}.
Notice that $K$ acts trivially on $U_x$ by minimality.
Then we have $\{0\} \neq L^2(U_x,E_x)^{G_x} \subset L^2(U_x,E_x)^K=L^2(U_x,E_x^K)$.
It follows that there is $v \in E_x^K \smallsetminus \{0\}$.
Let $y \in U_x$ and denote by $W_y'\cong G_x \times_{G_y} U_y' \subset U_x$ a tube around $y$ in $U_x$. 
Denote by $W_{y(K)}'$ the principal orbit bundle of $W_y'$ that is the dense open subset of $W_y'$
given by the points with stabilizer conjugated with $K$ in $G_x$. 
Each point of $W_{y(K)}'$ has a neighborhood of the form $G_x/K \times V$ with $K$ acting trivially on $V$.
Let $s \in \maC(G_x/K \times V, E_x)$ be given by $s([g],z)=gv$.
The section $s$ does not depend on the representative of $[g]$ in $G_x$ because $v \in E_x^K$
and is clearly $G_x$-invariant.
If $f\in \maC_c(G_x/K \times V)^{G_x}$ is any compactly supported function then
$f s \in \maC_c(G_x/K \times V, E_x)^{G_x} \subset L^2(W_y',E_x)^{G_x}$.
Now if $W_y\cong G \times_{G_y} U_y$ is a tube around $y$ in $W_x$ then assuming $U_y'$ small enough we have
$G \times_{G_x} W_y' \subset  G \times_{G_y} U_y$ and thus 
$0\neq L^2(W_y',E_x)^{G_x} \cong L^2(G \times_{G_x} W_y',E)^G \subset L^2(W_y,E)^G$.
It follows that $y \in M \smallsetminus \maN$ and therefore $W_x \subset M\smallsetminus \maN$. 
In other words, $M\smallsetminus \maN$ is open. This complete the proof.

\end{proof}

\begin{remark}\label{rem.N.Fredholm}
Let $\maN_\alpha=\big\{x\in M, \ \exists W_x,\ \text{such that } L^2(W_x,E)_\alpha=\{0\}\big\}$ be the clopen defined in 
Lemma \ref{lem.L2.tube.non.zero}. We have $L^2(\maN_\alpha,E)_\alpha=\{0\}$ 
and therefore there is no condition on $L^2(\maN_\alpha,E)_\alpha$ for an operator to be Fredholm.
By a discussion similar to the one of Lemma \ref{lemma.pathological:points}, 
we see that $P\in \Psi_M(\maH)^G$ is such that $\pi_\alpha(P)$ is Fredholm if, and only, if
$\chi_{M\smallsetminus \maN_\alpha}P \chi_{M\smallsetminus \maN_\alpha}$ is Fredholm, where $\chi_{M\smallsetminus \maN_\alpha}$ 
denotes the characteristic function of $M\smallsetminus \maN_\alpha$.

Moreover, we see that if $\maN_\alpha$ is not empty then the image of $\maC(\maN_\alpha)^G$ in $\maL(\maH_\alpha)$
is $0$, i.e. $M_\phi=0$, $\forall \phi \in \maC(\maN_\alpha)^G$. 
It follows that the image of $\maC(M)^G$ is the same as the image of $\maC(M\smallsetminus \maN_\alpha)^G$. 

For example, let $G=SO(3)$, let $M=S^2_1 \sqcup S^2_1$ be the disjoint union of 
two spheres $S^2 \subset \RR^3$ with a trivial action on $S_1^2$ and the induced action from $\RR^3$ on $S_2^2$.
Let then $E=M \times \CC^3$ with the natural action on $\CC^3$. Then
$L^2(M,E)^G=L^2(S_2^2,\CC^3)^G \neq 0$. Indeed, $(\CC^3)^G=0$ because $\CC^3$ is an irreducible 
representation of $SO(3)$ therefore $L^2(S_1^2,\CC^3)^G=L^2(S_1^2,(\CC^3)^G)=0$. 
Moreover, the function $f(x)=x \in \RR^3 \subset \CC^3$ is $G$-invariant and belongs to $L^2(S_2^2,\CC^3)^G$.
Now if $\chi_{S^2_1}$ is the characteristic function of $S_1^2$ then $M_{\chi_{S^2_1}} : L^2(M,E)^G \rightarrow L^2(M,E)^G$ is zero. 
\end{remark}

\begin{lemma}\label{lemma.cap.compact}
Let $\maP$ be the clopen introduced in Equation \eqref{eq.def.maP} 
and let $\maN_\alpha$ be the clopen introduced in Equation \eqref{eq.lem.N}.
Let $\alpha \in \widehat{G}$ and let $\maH_\alpha=L^2(M,E)_\alpha\cong \alpha \otimes L^2(M,E\otimes \alpha^*)^G$.  
Let $f\in \maC(M\smallsetminus (\maP\cup \maN_\alpha))^G$ then $M_f \in \maK(\maH_\alpha)$ if, and only, if $f=0$.
\end{lemma}

\begin{proof}
We may assume $\alpha$ to be the trivial representation $\one \in \widehat{G}$. 
Let $f \in \maC(M\smallsetminus (\maP\cup \maN_\alpha))^G$ be non zero such that $M_f $ is compact on $\maH_\one=L^2(M,E)^G$. 
Let then $W_x \cong G \times_{G_x} U_x $ be a tube on which $|f| >\varepsilon>0$. 
Denote by $\chi_{x}$ the characteristic function of $W_x$ and let $\widetilde{M_f}$ be the restriction of 
$M_f$ to $L^2(W_x,E)^G\cong L^2(U_x,E_x)^{G_x}$. Then $\widetilde{M_f}$ is invertible with inverse $M_{f\chi_{x}}$.
Thus by Banach open mapping theorem $\widetilde{M_f}$ is open but also compact therefore $L^2(U_x,E_x)^{G_x}$ is finite dimensional.
Notice that $L^2(U_x,E_x)^{G_x}\neq 0$ since $x \in M\smallsetminus \maN_\alpha$.    
Let then $s \in  L^2(U_x,E_x)^{G_x}$ be non zero.
For any $n \in \NN$, we can certainly find $n+1$ disjoint $G_x$-invariant annulus in $U_x$
such that the restriction of $s$ to each of this annulus is non zero
because the action is isometric and $\dim U_x >0$ since $x \in M \smallsetminus \maP$.
By considering the characteristic functions $\chi_i$ of this $n+1$ annulus, 
we get $n+1$ linearly independent functions $\chi_i s$, 
thus a contradiction.   
\end{proof}

Recall that we denote by $\maH=L^2(M,E)$.
Let $\alpha \in \widehat{G}$, let $P\in \maL(\maH)^G$ 
and recall that $\pi_\alpha(P) : \maH_\alpha \rightarrow \maH_\alpha$ is
the restriction of $P$ to the $\alpha$-isotypical component 
$\maH_\alpha \cong \alpha \otimes (\alpha^* \otimes \maH)^G$ of $\maH$.

\begin{proposition}[Simonenko's equivariant localization principle]
\label{prop.simonenko.equivariant}
Let $M$ be a closed $G$-manifold as before.
Let $\maP$ be as in Equation \eqref{eq.def.maP}
and let $\alpha \in \widehat{G}$.
Let $P \in \Psi_M^G(\maH)=\{P\in \maL(\maH)^G | [P,M_\phi] \in \maK(\maH), \forall \phi \in \maC(M)\}$. 
Then $P$ is locally $\alpha$-invertible on $M\smallsetminus \maP$ (see Definition \ref{def.alpha.invertible}) if, and
only if, $\pi_\alpha(P)$ is Fredholm.
\end{proposition}

\begin{proof}
Let $\maN_\alpha$ be as in Equation \eqref{eq.lem.N}.
Notice that any operator is locally $\alpha$-invertible at $x \in \maN_\alpha$ as operator between the null vector space.
Similarly, on $\maN_\alpha$ the operator $\pi_\alpha(P)$ is Fredholm.

By Lemma \ref{lemma.pathological:points} and Remark \ref{rem.N.Fredholm}, 
we may replace $M$ with $M\smallsetminus (\maP \cup N)$ 
and assume that for any $x\in M$, $(T^*_GM)_x$ is not reduce to $\{0\}$ 
and that there is a tube $W_x$ around $x$ such that $L^2(W_x,E_x)_\alpha\neq \{0\}$. 
Under this hypothesis, we have that $\maC(M/G) \rightarrow \maL(\maH_\alpha)$ is faithful, non degenerate 
and does not intersect $\maK(\maH_\alpha)\smallsetminus \{0\}$, 
see Lemma \ref{lemma.cap.compact} 
and Notation and Hypothesis \ref{NotHyp}.
Moreover, $\maC(M/G) \rightarrow \maL(\maH_\alpha)$
has the property of strong convergence to $0$, see  Definition \ref{def.property.strong.CV}. 
Indeed, this is equivalent to say that the volume 
of the slice at $x$ goes to zero when it becomes small.
Let us now introduce the $C^*$-algebra $\Psi_{M/G}(\maH_\alpha)$ defined in \ref{not.Simonenko}.
Clearly, $\pi_\alpha(\Psi_M^G(\maH))$ is a sub-$C^*$-algebra of $\Psi_{M/G}(\maH_\alpha)$. 
Therefore, $\pi_\alpha(P) \in \pi_\alpha(\Psi_M^G(\maH)) \subset \Psi_{M/G}(\maH_\alpha)$ is Fredholm if, and only if, 
it is locally invertible on $M/G$.
By definition, $P$ is locally $\alpha$-invertible at $x\in M$ if, and only, if $\pi_\alpha(P)$ is locally invertible at $Gx \in M/G$, thus the result follows from Proposition \ref{prop.simonenko.gen}.  
%
\end{proof}

\begin{theorem}\label{thm.main.Simonenko}
Let $M$ be a closed $G$-manifold as before
and let $\maP$ be as in Equation \eqref{eq.def.maP}.
Let $P\in\psi^m(M;E_0,E_1)^G$ and $\alpha \in
\widehat{G}$. Then the following are equivalent:
\begin{enumerate}
  \item $\pi_\alpha(P) : H^s(M;E_0)_\alpha \to H^{s-m}(M;E_1)_\alpha$ is
    Fredholm for any $s \in \RR$,
  \item $P$ is transversally $\alpha$-elliptic (see Definition \ref{def.chi.ps}),
  \item $P$ is locally $\alpha$-invertible on $M\smallsetminus \maP$ (see Definition \ref{def.alpha.invertible}).
\end{enumerate}
\end{theorem}
\begin{proof}
The first equivalence is given by Theorem \ref{thm.main1}.  
Now Proposition \ref{prop.simonenko.equivariant}
implies that \textit{(1)} is equivalent to \textit{(3)}.
 \end{proof}

 In particular, we obtain the
following consequence of the localization principle.

\begin{corollary}
Let $P\in \op{m}^G$ be a $G$-transversally elliptic operator, see section \ref{sub.psdo}. Then $P$ is
locally $\alpha$-invertible on $M\smallsetminus \maP$ for any $\alpha \in \widehat{G}$, as in
Definition \ref{def.alpha.invertible}.
\end{corollary}

\begin{proof}
Using Theorem \ref{Thm.atiyahGell} we obtain that $\pi_\alpha(P)$ is
Fredholm.  Therefore by Proposition \ref{prop.simonenko.equivariant}
$P$ is locally $\alpha$-invertible.
\end{proof}



\setlength{\baselineskip}{4.75mm}
\bibliographystyle{plain}

\bibliography{ref_Simonenko}

\end{document}